\documentclass[letterpaper,12pt,twoside,reqno]{amsart}
\usepackage[top=1.2in,bottom=1.2in]{geometry}
\usepackage{times}
\usepackage{amsmath}
\usepackage{amssymb}
\usepackage{amsfonts}
\usepackage{amscd}
\usepackage{amsthm}
\usepackage{amsxtra}
\usepackage{stmaryrd}

\usepackage[all]{xy}  \CompileMatrices
\usepackage{mathrsfs}
\usepackage{nicefrac}
\usepackage{graphicx}
\usepackage{color}
\usepackage[raggedright,IT,hang]{subfigure}
\usepackage{wrapfig}

\DeclareMathOperator{\direc}{lin} % directions
\newcommand{\Cplx}{\mathcal C}
\newcommand{\cplxD}{\mathcal D}

\newcommand{\TiltCplx}{\mathcal S}
\newcommand{\TTcplx}{\mathcal T}
\newcommand{\Pfaces}{\mathfrak F}

\DeclareMathOperator{\del}{dl}

\renewcommand{\paragraph}[1]{\par\vspace{1ex}\noindent #1}

\newcommand{\todo}[1]{}

\theoremstyle{plain}
\newtheorem{theorem}{Theorem}

\newtheorem*{proposition*}{Proposition}

\newtheorem*{corollary*}{Corollary}
\newtheorem{lemma}[theorem]{Lemma}

\newtheorem*{theorem*}{Theorem}
\newtheorem*{lemma*}{Lemma}
\newtheorem*{conjecture*}{Conjecture}

\newtheorem*{question*}{Question}

\theoremstyle{definition}

\newtheorem*{exercise*}{Exercise}

\theoremstyle{remark}
\newtheorem{remark}[theorem]{Remark}
\newtheorem*{remark*}{Remark}
\newtheorem{remsTh}[theorem]{Remarks}
\newtheorem{claim}[theorem]{Claim}

\newcommand{\subclass}[1]{}

\newcommand{\enumTi}[1]{\renewcommand{\theenumi}{#1}}

\newcommand{\alphenumi}{\enumTi{\alph{enumi}}}
\newcommand{\Alphenumi}{\enumTi{\Alph{enumi}}}
\newcommand{\romenumi}{\enumTi{\roman{enumi}}}
\newcommand{\itromenumi}{\enumTi{\textit{\roman{enumi}}}}
\newenvironment{aenumeratei}{%
  \begin{enumerate}\alphenumi}{%
  \end{enumerate}}

\renewcommand{\em}{\sl}

\DeclareMathOperator{\id}{id}

\DeclareMathOperator{\relint}{relint}

\DeclareMathOperator{\img}{im}

\DeclareMathOperator{\aff}{aff}
\DeclareMathOperator{\conv}{conv}

\newcommand{\orth}{\bot}

\newcommand{\One}{\mathbf{1}}

\newcommand{\place}{{\text{\footnotesize$\oblong$}}}

\newcommand{\lt}{\left}
\newcommand{\rt}{\right}

\newcommand{\Tp}{{\mspace{-1mu}\scriptscriptstyle\top\mspace{-1mu}}}
\newcommand{\polar}{\vartriangle}
\newcommand{\poOps}{\Diamond}

\newcommand{\restr}[2]{\lt.{#1}\vrule{}_{#2}\rt.}

\newcommand{\sstack}[1]{{\substack{#1}}}

\newcommand{\abs}[1]{{\lt\lvert{#1}\rt\rvert}}

\newcommand{\sabs}[1]{{\lvert{#1}\rvert}}

\newcommand{\close}[1]{\overline{#1}}

\DeclareMathOperator{\dl}{dl}

\newcommand{\nfrac}[2]{{\nicefrac{#1}{#2}}}

\newcommand{\RR}{\mathbb{R}}

\newcommand{\ZZ}{\mathbb{Z}}

\newcommand{\eps}{\varepsilon}

\newcommand{\iprod}{\bullet}

\newlength{\algotabbingwidth}
\numberwithin{theorem}{section}

\renewcommand{\iprod}{\mathop{\cdot}}

\newcommand{\Vset}[1]{V_{#1}}
\newcommand{\Vsetn}{{\Vset{n}}}
\newcommand{\Eset}[1]{{E_{#1}}}
\newcommand{\Esetn}{\Eset{n}}

\newcommand{\Stsp}[1]{S_{#1}}
\newcommand{\Stspn}{\Stsp{n}}
\newcommand{\Gtsp}[1]{P_{#1}}
\newcommand{\Gtspn}{\Gtsp{n}}

\newcommand{\Lsp}{L}

\begin{document}

\title[On the facial structure of STSP and GTSP]{On the facial structure of Symmetric and Graphical Traveling Salesman Polyhedra}%
\author{Dirk Oliver Theis}%
\address{Dirk Oliver Theis, %
  Service de G\'eom\'etrie Combinatoire et Th\'eorie des Groupes (CP 216), %
  D\'epartement de Math\'ematique, %
  Universit\'e Libre de Bruxelles, %
  Bd du Triomphe,
  1050 Brussels, Belgium}%
%\curaddr{}
\email{Dirk.Theis@ulb.ac.be}%
\thanks{Research supported by \textit{Deutsche Forschungsgemeinschaft (DFG)} as project RE~776/9-1.  Author supported by \textit{Communaut\'e fran\c caise de
    Belgique -- Actions de Recherche Concert\'ees.}}
% \subjclass[2000]{Primary 52B12 Secondary 52B11}

\date{Thu Apr~9 21:27:07~EDT 2009}
%
%\date{}

%\dedicatory{}

\begin{abstract}
  The Symmetric Traveling Salesman Polytope $S_n$ for a fixed number $n$ of cities is a face of the corresponding Graphical Traveling Salesman Polyhedron $P_n$.  
  This has been used to study facets of $S_n$ using $P_n$ as a tool.  
  In this paper, we study the operation of ``rotating'' (or ``lifting'') valid inequalities for $S_n$ to obtain a valid inequalities for $P_n$.
  
  As an application, we describe a surprising relationship between (a) the parsimonious property of relaxations of the Symmetric Traveling Salesman Polytope and
  (b) a connectivity property of the ridge graph of the Graphical Traveling Salesman Polyhedron.

\end{abstract}
\maketitle

\setcounter{tocdepth}{2}
\tableofcontents

% subdiv/intro.tex

\section{Introduction}

Suppose that $S$ and $P$ are polyhedra, and that $S$ is a proper face of $P$.  If $a\iprod x \ge \alpha$ is a valid inequality for $S$, it can be ``rotated'' so
that it becomes also valid for $P$.  By ``rotation'' we mean modifying left and right hand sides of the inequality in such a way that the set of points in the
affine hull of $S$ which satisfy the inequality with equation remains the same, yet the hyperplane the inequality defines in the ambient space changes.
Technically, this amounts to adding equations to $a\iprod x \ge \alpha$, which are valid for $S$.

Once the inequality is rotated so that it is valid for $P$, one may ask which face of $P$ is defined by the rotated inequality.  Since $S\ne P$, there is never
only one such face, but even when we aim for inclusion-wise maximal faces of $P$ defined by some rotated version of $a\iprod x \ge \alpha$, in general, these are
not unique either.

Rotation is a standard tool in Discrete Optimization.  The most prominent example is probably (sequential) lifting, which is a constrained form of rotation.  In
this setting, $P$ is a polyhedron for which the non-negativity inequality $x_j \ge 0$ for a coordinate $j$ is valid, defining a non-empty face $S := P\cap \{x\mid
x_j=0\}$.  Then, an inequality valid for $S$ is rotated by adding scalar multiples of the equation $x_j=0$ to it in such a way that it becomes valid for $P$ and
the face defined by the rotated inequality is strictly greater than the face of $S$ defined by it.  By iterating this procedure, one may ``sequentially'' lift
inequalities which are valid for a smaller face $S$, which is an intersection of the faces defined by the non-negativity inequalities for a set of coordinates.
The face of $P$ defined by the sequentially lifted inequality may in general depend on the order in which the coordinates are processed.  The same procedure works
when generic inequalities $c\iprod x \ge \gamma$ are used instead of the non-negativity inequalities.  

Sequential lifting or other rotation-based tools are applied manually to find facets of polyhedra which contain faces which are better understood.  Often, the
faces are ``smaller versions'' of the original polyhedron.  Moreover, mechanisms of this kind are used computationally in cutting-plane algorithms where some
cutting-plane generation procedure first works on a face and then has to lift the obtained inequalities.

\paragraph{In this paper,} %
we study what rotating inequalities does for the Symmetric Traveling Salesman Polytope and the Graphical Traveling Salesman Polyhedron.  Let $n\ge 3$ be an
integer, $\Vsetn := \{1,\dots,n\}$ and $\Esetn$ be the set of all unordered pairs (two-element subsets) $\{i,j\} \in \Vsetn$, i.e., the set of edges of the
complete graph with vertex set $\Vsetn$.  The two polyhedra are subsets of the space $\RR^\Esetn$ of vectors indexed by the elements of $\Esetn$.  The Symmetric
Traveling Salesman Polytope $S_n$ is the convex hull of all incidence vectors of edge sets of circles with vertex set $\Vsetn$ (or, if you prefer, of Hamilton
cycles of the complete graph $K_n$).  The Graphical Traveling Salesman Polyhedron $P_n$ is the convex hull of all vectors corresponding to connected Eulerian
multi-graphs with vertex set $\Vsetn$.  (The precise definitions will be given below.)

Since the seminal work of Naddef \& Rinaldi \cite{NadRina91,NadRina93} on these two polyhedra, it is known that the former is a face of the latter.  Moreover,
Naddef \& Rinaldi proved a theorem which, in our terminology, says that, if an inequality defines a facet of $S_n$, then there is a unique maximal face of $P_n$
which can be obtained by rotating the inequality, and this maximal obtainable face is a facet of $P_n$.

Naddef \& Rinaldi managed to classify the facets of $P_n$ into tree types: non-negativity facets, degree facets, and the rest, called TT-facets.  While the degree
facets and non-negativity facets are both small in number and easily understood, the interesting class both for understanding the polyhedron and for applications
is the huge set of TT-facets.  By the theorem just mentioned, once one knows that the degree facets of $P_n$ are precisely those which contain $S_n$ --- also an
achievement of Naddef \& Rinaldi's paper ---, this also classifies the facets of $S_n$ into two types: non-negativity and TT-facets.  Again, for applications in
Discrete Optimization, the TT-facets are the important ones.

Not so long ago, Oswald, Reinelt and Theis \cite{OswReiThe05,OsRlTheis07} have refined the classification by splitting the TT-facets of $P_n$ into two subclasses:
NR-facets and non-NR-facets, depending on whether the intersection of the facet with $S_n$ is a facet of $S_n$ (these $P_n$ facets are called NR-facets) or a face
of $S_n$ of smaller dimension (these are called non-NR-facets).
The main difficulty in this sub-classification was showing that the non-NR class is not empty.  The existence of non-NR-facets has some unpleasant consequences
both for theoretical research and practical computational approaches to solving Traveling Salesman Problem instances.  On the theoretical side, it is much easier
to prove facet-defining property of inequalities for $P_n$ than for $S_n$.  Moreover, $P_n$ pleasantly preserves facet-defining property when a certain important
lifting operation for facet-defining inequalities (which replaces vertices by sets of vertices) is performed.  For $S_n$, this is not known to be true.  On the
computational side, in the context of cutting-plane methods for $S_n$, certain generic separation algorithms produce inequalities which are facet-defining for
$P_n$, but sometimes it is not clear whether these inequalities must be strengthened if they are to define facets of $S_n$.  Examples of such separation algorithms
include the local cuts method of Applegate, Bixby, Chv\`atal \& Cook \cite{ABCC2001,ABCC03,ABCCbook2006} (see the discussion in \cite{OsRlTheis07}) or the
path-lifting method of Carr \cite{Carr2004}.

In terms of rotation, the result in \cite{OswReiThe05,OsRlTheis07} shows that there are valid inequalities for $S_n$ which do not define facets of $S_n$, but which
can be rotated to define facets of $P_n$.  The starting point of this paper is the question what properties these valid inequalities for $S_n$ might have.  The
results we propose are most easily formulated using the terminology of polar polyhedra.  A polar polyhedron $S^\polar$ of a polyhedron $S$ has the property that
the points of $S^\polar$ are in bijection with the linear inequalities (up to scaling) for $S$.  Moreover, a point $a$ is contained in a face of dimension $k$ of
$S^\polar$, if, and only if, the corresponding inequality defines a face of dimension at least $\dim S-k$ of $S$.  In particular, the vertices of $S^\polar$ are in
bijection with the facets of $S$.  Also recall the concept of a polyhedral complex: a (finite) set of polyhedra, closed under taking faces, such that the
intersection of any two polyhedra in the set is a face of both.

We have results about the ``interesting'' part of the polar of $S_n$, namely the part which remains if we take only those faces of the polar, which do not contain
a vertex corresponding to a non-negativity facet of $S_n$.  Informally, this corresponds to taking only the TT-class of valid inequalities for $S_n$ (the
correspondence will be made precise later).

This subset of faces of the polar of $S_n$ is a polyhedral complex; let us denote it by $\Cplx$ for a moment.  Take a point in $\Cplx$, consider the corresponding
valid inequality for $S_n$, and rotate it.  A certain set of faces of $P_n$ can be defined by the rotated versions of this inequality.  Now we partition the points
contained in $\Cplx$ in the following way: two points are in the same cell of the partition, if, by rotating the corresponding valid inequalities, the two sets of
faces of $P_n$ which can be defined coincide.

In fact, the partition whose definition we have just outlined, gives a polyhedral subdivision $\TiltCplx$ of $\Cplx$, i.e., the set of closures of the cells is a
polyhedral complex, and every face of $\Cplx$ is a disjoint union of cells.  Indeed, this is true in the general situation when a polytope $S$ is a face of another
polytope $P$, and such a polyhedral subdivision is called a rotation complex.  In the TSP situation, we can say more:
\begin{enumerate}\Alphenumi
\item\label{intro:mainresutls:charac}%
  The decomposition of $\Cplx$ into cells can be described in a natural way that does not refer to rotation; in fact, it does not refer to the Graphical Traveling
  Salesman at all.  Indeed, for a point $a$ contained in $\Cplx$, it suffices to check the sign of all the expressions $a_{uv} - a_{uw} - a_{wv}$, with $u,v,w$
  three distinct vertices in $\Vsetn$.  (As customary, we use the abbreviated notation $uv := \{u,v\}$.)
\item\label{intro:mainresutls:inj}%
  The points in $\Cplx$ are in bijection with the ``important'' part of the polar of $P_n$ (the definition of polar here is not canonical and will be made
  precise), and this bijection maps faces of the polar of $P_n$ onto faces of the rotation complex $\TiltCplx$.  In other words, the polar of $P_n$ can be
  ``flattened'' onto the polar of $S_n$.
\end{enumerate}

Again, ``important'' is meant to be understood in the sense that it corresponds to considering TT-type inequalities only.

Recall that the common refinement of two polyhedral complexes is the set of all intersections of polyhedra in the two complexes.
Item~(\ref{intro:mainresutls:charac}) can be restated as saying that the rotation complex $\TiltCplx$ is the common refinement of $\Cplx$ with a natural projection
of the metric cone.  (The metric cone consists of all functions $\Esetn\to\RR_+$ satisfying the triangle inequality).  Note that the occurrence of the metric cone
in the context of the two polyhedra $S_n$ and $P_n$ is no surprise: it is known that $P_n$ is the intersection of the \textit{positive orthant} $\RR_+^\Esetn$ with
the Minkowski sum of $S_n$ and the dual of the metric cone~\cite{TheisStspGtspMinkowski}.  Item~(\ref{intro:mainresutls:inj}) addresses the uniqueness question for
faces defined by rotated inequalities addressed above.  Note, though, that having a point-wise bijection is a stronger statement than saying that the maximal faces
obtainable by rotation are unique.\todo{Not sure}

\paragraph{We believe these results to be of interest in their own right,} %
because they clarify the relationship between the valid inequalities for $S_n$ and $P_n$.  Having said that, in this paper, we apply them to a problem concerning
the ridge graph of $P_n$.  The ridge graph has as its vertices the facets, and two facets are linked by an edge if and only if their intersection is a ridge, i.e.,
a face of dimension $\dim P_n - 2$.  The ridge graph is of certain importance for the problem of computing a complete system of facet-defining inequalities, when
the points and extreme rays are given.  A common solution here is to search in the ridge graph, i.e., once a facet is found, its neighbors are computed.  A problem
which may occur is that, for some facets, computing the neighbors is not feasible (given the power of current computer systems).  Due to the connectivity of the
ridge graph, some of its vertices are allowed to be dead ends in the search, and still all vertices are reached by the search.  For example, when the facets of a
$d$-dimensional polytope are computed in this way, by Balinski's Theorem, one may omit $d-1$ arbitrarily selected facets from the search, and still reach all other
facets.  Very often, however, the number of facets whose neighbors cannot be computed is too large (exponential in the dimension).  Thus, one would like to prove
connectivity properties of the ridge graph which allow for these vertices to be dead ends in the search.

Our result on the ridge graph of $P_n$ states the following: If a system of NR-facet-defining inequalities satisfies the so-called parsimonious property
\cite{GoemBertim93,Goemans95}, the removal of the corresponding vertices from the ridge graph leaves connected components, each of which contains a vertex
corresponding to an NR-facet.  The proof of this makes use of~(\ref{intro:mainresutls:inj}) above in an essential way.  The statement has been used to prove the
completeness of an outer description for $P_9$ in \cite{OsRlTheis07} in the scenario sketched above.

\paragraph{This paper is organized as follows.} %
In the short second section, we shall define some basic concepts from polyhedral theory.  In Section~\ref{sec:expo}, will provide rigorous formulations of all
of our results.  Section~\ref{sec:tsp} contains the proofs of the results about the rotation complex, while the results about the ridge graph are proved in
Section~\ref{sec:parsi}.

We will need to make use of linear-algebraic and polyhedral ideas quite heavily.  Although we give all the relevant definitions, understanding this paper will be a
piece of hard work if one is not at ease with the theory of polyhedra, polarity, projective transformations, and polyhedral complexes, as laid out in the relevant
chapters of either \cite{Grunbaum03} or \cite{ZieglerLectPtp}.

\section{Some basic definitions and notations}\label{sec:prelim}

\subsubsection{Euclidean space notations}
%% Inner product things
We denote by $x\iprod y$ the standard scalar product in $\RR^m$.  For a linear subspace $L\subset \RR^m$, denote by $L^\orth := \{q\in\RR^m\mid q\iprod
x=0\;\forall x\in L\}$ the orthogonal complement of $L$.

%% Affine hull
For $X\subset \RR^m$, we denote by $\aff X$ the affine hull of $X$, i.e., the smallest affine subspace of $\RR^m$ containing $X$.  We let $\direc X$ denote the
linear space generated by the points $y-x$, $x,y\in X$.  Hence, $\aff X = x + \direc X$ holds for every $x\in \aff X$.

%% Topology
%\subsubsection{}%
For $X\subset \RR^m$, we denote by $\close X$ the closure of $X$ in the topological sense.  The \textit{relative interior} $\relint P$ of a polyhedron $P$ is the
interior (in the topological sense) of $P$ in the affine space spanned by $P$, in other words, $\relint P = P \setminus \bigcup_{F\subsetneq P} F$, where the union
runs over all faces of $P$.
The \textit{boundary} of a polyhedron is $\partial P := P\setminus \relint P = \bigcup_{F\subsetneq P} F$ where the union runs over all faces of $P$.

\subsubsection{Projective mappings}
An mapping $g$ between vector spaces is called affine if there exists a constant (vector) $a$ such that $g-a$ is linear.
A mapping $f\colon L \to L'$ between two vector spaces is called \textit{projective,} if there exists a linear mapping
\begin{equation*}
  \begin{pmatrix}
    f_{00} & f_{01} \\
    f_{10} & f_{11} \\
  \end{pmatrix}
  = \tilde f\colon \RR\times L \to \RR\times L'
\end{equation*}
decomposable into linear mappings $f_{00}\colon\RR\to\RR$, $f_{01}\colon L \to\RR$, $f_{10}\colon \RR \to L'$, $f_{11}\colon L \to L'$, such that $f(x) = P(\tilde
f(1,x))$, with the shorthand $P(t,x) := x/t$.  Informally, we say that $f$ can be ``written as a linear mapping'' $\tilde f$.  Using matrices, $f_{00}$ can be
identified with a real constant, $f_{10}$ with a column-vector and $f_{01}$ with a row-vector.

\begin{remark}\label{rem:prelim:prj-map-concat}
  When $f$ and $g$ are projective mappings which can be written as linear mappings $\tilde f$ and $\tilde g$, respectively, then $f\circ g$ can be written as
  $\tilde f \circ \tilde g$.
\end{remark}

%% Polyhedral complexes
\subsubsection{Polyhedral complexes.}  %
A \textit{polyhedral complex} is a set of polyhedra $\Cplx$ with the properties that (a) if $F\in\Cplx$ and $G$ is a face of $F$, then $F\in \Cplx$; and (b) if
$F,G \in \Cplx$, then $F\cap G$ is a face of both $F$ and $G$.  The polyhedra in $\Cplx$ are called the faces of $\Cplx$, and faces of a $\Cplx$ having dimension
$0$ (or $1$, respectively) are called vertices (or edges, respectively) of $\Cplx$.  A \textit{sub-complex} of a polyhedral complex $\Cplx$ is a polyhedral
complex $\cplxD$ with $\cplxD \subset \Cplx$.

For a polyhedral complex $\Cplx$, we denote by $\sabs{\Cplx} := \bigcup_{F\in \Cplx} F$ its \textit{underlying point set}, and, informally, we say that a point $x$
is in $\Cplx$, if $x\in\sabs{\Cplx}$.  

For a polyhedron $P$, let $\Cplx(P)$ be the set of all of its faces.  This is a polyhedral complex with underlying point set $P$.  Moreover, we let $\bar\Cplx(P)$
be the polyhedral complex of all bounded faces of $P$.

For a polyhedral complex $\Cplx$ and a set of faces $\cplxD \subset \Cplx$, we define the deletion of $\cplxD$ in $\Cplx$ to be the polyhedral sub-complex of
$\Cplx$ consisting of all faces $F\in\Cplx$ whose intersection with all faces in $\cplxD$ is empty:
\begin{equation*}
  \dl(\cplxD,\Cplx) := \bigl\{ F \in \Cplx \bigm| \forall G \in \cplxD\colon F\cap G = \emptyset  \bigr\}
\end{equation*}

%% subdivision
Let $\Cplx$ and $\cplxD$ be two polyhedral complexes.  $\cplxD$ is called a \textit{subdivision} of $\Cplx$, if, (a) every face of $\cplxD$ is contained in some
face of $\Cplx$; and (b) every face of $\Cplx$ is a union of faces of $\cplxD$.

%% common refinement
Let $\Cplx$ and $\cplxD$ be two polyhedral complexes.  The \textit{common refinement} of $\Cplx$ and $\cplxD$ is the polyhedral complex whose faces are all the
intersections of faces of $\Cplx$ and $\cplxD$: $\Cplx \vee \cplxD := \{ F\cap G \mid F \in \Cplx, G\in\cplxD \}$.  The common refinement $\Cplx \vee \cplxD$ is a
subdivision of both $\Cplx$ and $\cplxD$.

%% refinement map
Let $\Cplx$ be a polyhedral complex, and $f\colon \sabs\Cplx \to \RR^k$ a mapping.  We say that $f$ \textit{induces the polyhedral complex $\cplxD$}, if, for every
$F\in\Cplx$, its image $f(F)$ under $f$ is a polyhedron, and the set of all these polyhedra is equal to (the polyhedral complex) $\cplxD$.  The following wording
is customary: If $\cplxD'$ is a polyhedral complex and $f\colon \sabs\Cplx \to \sabs{\cplxD'}$ is a homeomorphism which induces a polyhedral complex $\cplxD$ which
is a subdivision of $\cplxD'$, then $f$ is called a \textit{refinement map.}  Two polyhedral complexes $\Cplx$ and $\cplxD$ are called combinatorially equivalent,
if there exists a bijection $\phi\colon \Cplx \to \cplxD$, which preserves the inclusion relation of faces, i.e., if $F\subset F'$ are two faces of $\Cplx$, then
$\phi(F)\subset\phi(F')$.  We say that a mapping $f\colon \sabs{\Cplx} \to \sabs{\cplxD}$ \textit{induces a combinatorial equivalence,} if $f$ induces the
polyhedral complex $\cplxD$.  In this case, $\Cplx$ and $\cplxD$ are combinatorially equivalent via the mapping $F\mapsto f(F)$.

%% fans
A polyhedral complex is a (pointed) \textit{fan} if it contains precisely one vertex, and each face which is not a vertex is empty or a
pointed cone.  A fan $\Cplx$ is \textit{complete,} if $\abs\Cplx$ is equal to the ambient space.

%% graph
The \textit{1-skeleton} or \textit{graph} of a polyhedral complex $\Cplx$ is the graph $G$ whose vertices are the vertices of $\Cplx$, with two vertices of $G$
being adjacent if and only if there exists an edge of $\Cplx$ containing them both.

For more on polyhedral complexes see the textbooks by Gr\"unbaum~\cite{Grunbaum03} or Ziegler~\cite{ZieglerLectPtp}.

\subsubsection{Miscellaneous.}  %
For a matrix $M$ we denote by $M^\Tp$ its transpose.  The restriction of a mapping $f\colon X\to Y$ to a set $Z\subset X$ is denoted by $\restr{f}{Z}$.

%%% Local Variables: 
%%% mode: latex
%%% TeX-master: "paper.tsp.tex"
%%% fill-column: 163
%%% End: 

% subdiv/expo.tex
\section{Exposition of results}\label{sec:expo}

Fix an integer $n\ge 3$.  The \textit{Symmetric Traveling Salesman Polytope} is defined as the convex hull in $\RR^\Esetn$ of all edge sets of circles with vertex
set $\Vsetn$ (or Hamiltonian cycles in the complete graph $K_n$):
\begin{equation}\label{eq:tsp:def-S}
  S_n := \conv\bigl\{  \chi^{E(C)} \bigm| C\text{ is the circle with } V(C)=\Vsetn \bigr\},
\end{equation}
where $\chi^F$ denotes the characteristic vector of a set $F$, i.e., $\chi^F_e=1$, if $e\in F$, and $0$ otherwise.
Ever since the mid nineteen-fifties, when a series of short communications and papers initiated the study of this family of polytopes
\cite{Heller55a,Heller55b,Heller56,Kuhn55,Norman55}, it has received steady research attention.  Apart from being of importance in combinatorial optimization for
solving the famous Traveling Salesman Problem, which consists in finding a shortest Hamilton cycle in a complete graph with ``lengths'' assigned to the edges (see,
e.g., \cite{ABCCbook2006,DaFulkJohn54,GroetschPadberg85,JuReiRin95,Gutin_Naddef,SchrijverBk03}), their combinatorial and linear-algebraic properties have been an
object of research.  For example, questions about aspects of the graph (1-skeleton) have been addressed \cite{Sierksma98}, particularly focusing on its diameter
\cite{Rispoli98, RispoCosar98, SierkTeunTiejss95, SierkTiejss92}, which is conjectured to be equal to two by Gr\"otschel \& Padberg \cite{GroetschPadberg85}.

The second polyhedron which we will consider is defined to be the convex hull of all edge multi-sets of connected Eulerian multi-graphs on the vertex set
$\Vsetn$:
\begin{equation}\label{eq:tsp:def-P}
  \begin{aligned}
    P_n := \conv\{&x\in\ZZ_+^\Esetn \mid \\
    &x\text{ defines a connected Eulerian multi-graph with vertex set } \Vsetn \},
  \end{aligned}  
\end{equation}
where we identify sub-multi-sets of $\Esetn$ with vectors in $\ZZ_+^\Esetn$ (i.e., there are $x_e$ copies of edge $e$ present in the multi-graph).  This polyhedron
was introduced in \cite{CornFonluNadd} under the name of \textit{Graphical Traveling Salesman Polyhedron} and has since frequently occurred in the literature on
Traveling Salesman Polyhedra.  It is particularly important in the study of properties, mainly facets, of Symmetric Traveling Salesman Polytopes (e.g.,
\cite{Goemans95,NaddefPochet01,NadRina91,NadRina93,NadRina07}, see \cite{ABCCbook2006,Gutin_Naddef} for further references).

With few exceptions (for example \cite{FonlNadd92,Norman55} for the case $n\le 5$; \cite{BoydCunningh91} for $n=6,7$; \cite{ChristofJungerReinelt91,
  ChristoReinelt96, ChristofReineltStsp10} for $n=8,9$), no complete characterization of the facets of $\Stspn$ or $\Gtspn$ are known.  In fact, since the
Traveling Salesman Problem is NP-hard, there cannot exist a polynomial time algorithm producing, for every $n$ and every point $x\in\RR^\Esetn$, a hyperplane
separating $x$ from $S_n$, unless $P$=$NP$.  Another noteworthy argument for the complexity of these polytopes is a result of Billera \& Saranarajan
\cite{BilleraSaranarajan96}: For every 0/1-polytope $P$, there exists an $n$ such that $P$ is affinely isomorphic to a face of $\Stspn$.

The polyhedron $\Gtspn$ has been called the \textit{Graphical Relaxation} of $\Stspn$ by Naddef \& Rinaldi \cite{NadRina91,NadRina93} who discovered and made use
of the fact that $\Stspn$ is a face of $\Gtspn$: While the latter is a full-dimensional unbounded polyhedron in $\RR^\Esetn$ \cite{CornFonluNadd}, the former is a
polytope of dimension $\binom{n}{2}-n$ \cite{Norman55}, and the inequality $\sum_{e\in\Esetn} x_e \ge n$ is valid for $\Gtspn$ and satisfied with equality only by
cycles, thus attesting to the face relation.

\subsection{Definitions of the polars}

From now on, assuming\footnote{%
  We choose $n\ge 5$ because otherwise the non-negativity inequalities do not define facets of $S_n$.%
} %%
$n\ge 5$ to be fixed, we will suppress the subscript in $\Stspn$ and $\Gtspn$ and just write $S$ and $P$.

% As mentioned above, $P$ is full-dimensional in $\RR^\Esetn$ \cite{CornFonluNadd}, and $S$ has dimension $\binom{n}{2}-n$ \cite{Norman55}.
%%
The set of facets of $P$ containing $S$ is known.  For $u\in\Vsetn$, let $\delta_u$ be the point in $\RR^\Esetn$ which is $\nfrac12$ on all edges incident to $u$
and zero otherwise.  It is proven in \cite{CornFonluNadd} that the inequalities $\delta_u\iprod x \ge 1$, $u\in\Vsetn$, define facets of $P$, the so-called
\textit{degree facets.}  Clearly, $S$ is the intersection of all the degree facets.

It is customary to write inequalities valid for $P$ in the form $a\iprod x \ge \alpha$, and we define the polars accordingly.  Define the linear space $\Lsp$ to be
the set of solutions to the $n$ linear equations $\delta_u\iprod x = 0$, $u\in\Vsetn$.  Note that the $\delta_u$ are linearly independent, $\dim S = \dim \Lsp$,
and the affine hull of $S$ is a translated copy of $\Lsp$.  Whenever $z$ is a relative interior point of $S$, the polar of $S$ may be defined as the following set:
\begin{equation}\label{eq:tsp:def-Spolar}
  S^\polar := \{  a\in\Lsp   \mid   (-a)\iprod (x-z) \le 1 \;\forall x\in S \}.
\end{equation}
So a point $a\in S^\polar$ corresponds to a valid inequality $a\iprod x \ge a\iprod z -1$ of $S$.  Changing $z$ amounts to submitting $S^\polar$ to a projective
transformation.  Although our results do not depend on the choice of $z$ (see \cite{subdivpolarface}), it makes things easier to define
\begin{equation}\label{eq:tsp:def-z}%
  z := \frac{2}{n-1} \One = \frac{1}{(n-1)!/2} \sum_{C} \chi^{E(C)} = \frac{2}{n-1} \sum_{u=1}^n \delta_u,
\end{equation}
where the first sum extends over all cycles with vertex set $\Vsetn$.  So $z$ is at the same time the average of the vertices $\chi^{E(C)}$ of $S$ and a weighted
sum of the left-hand sides $\delta_u$ of the equations.

Next, we construct a kind of polar for $P$.  For this, we might just intersect the polar cone $C:=\{(\alpha,a)\in\RR\times\RR^{\Esetn} \mid a\iprod
x\ge\alpha\;\forall x\in P\}$ with the hyperplane $\alpha+\sum_e a_e =1$.  From the observation \cite{CornFonluNadd} that $P$ is the Minkowski sum of
$\RR_+^\Esetn$ with a finite set of points in $\RR_+^\Esetn$, we see that this hyperplane intersects all extreme rays of $C$ except for $\RR_+(\alpha,0)$ which
does not correspond to a facet of $P$.  However, for our needs, it will be better to define $P^\polar$ to be a polyhedron which is projectively isomorphic to the
one we have just described:
\begin{equation*}
  P^\polar := \{ a\in \RR^m \mid a\iprod x\ge 1 \;\forall x\in P\}.
\end{equation*}

This set is sometimes called the \textit{blocking polyhedron} of $P$.  Calling it the \textit{polar (polyhedron)} of $P$ is justified by that fact that,
essentially, it has the defining properties of a polar polytope.  Let us elaborate.  For a face $F$ of $P$, define its \textit{conjugate face} $F^\poOps$ to be the
set of points $a \in P^\polar$ satisfying $a\iprod x = 1$ for every $x\in F$.
For brevity, we say that a face $F$ of $P$ is \textit{\label{txt:def:tsp:good-face}good} if it is not contained in a \textit{non-negativity facet,} i.e., a
facet defined by $x_e\ge 0$ (these inequalities do define facets of $P$ \cite{CornFonluNadd}).
Note that $P^\polar \subset\RR_+^\Esetn$, so the non-negativity inequalities are also valid for $P^\polar$, and hence $P^\polar$ has non-negativity faces.  (They
are possibly empty.)

\begin{lemma}\label{lem:prop-blockingpoly}
  The polar $P^\polar$ of $P$ has the following properties.
  \begin{enumerate}\alphenumi
  \item\label{prop:PRELIM:blocking-polarity:b} %
    Let $a\in\RR^\Esetn\setminus\{0\}$ and $d\ge-1$.  Then $a$ is a relative interior point of a non-trivial face of $P^\polar$ with co-dimension $d+1$ if and
    only if $(a,1)$ is valid for $P$ and defines a face of dimension $d$ of $P$.
  \item\label{prop:PRELIM:blocking-polarity:c} %
    Let $N \subset \Cplx(P)$ be the set of intersections of non-negativity facets $P$, and similarly $N' \subset \Cplx(P^\polar)$ be the set of all intersections
    of non-negativity faces of $P^\polar$.  Then conjugation of faces
    $\Cplx(P)\setminus N \to \Cplx(P^\polar)\setminus N'$,
    $F\mapsto F^\poOps := \{ a \in P^\polar \mid a\iprod x = 1 \;\forall x\in F\}$
    defines an inclusion reversing bijection.
  \item\label{prop:PRELIM:blocking-polarity:d} %
    A face $F$ of $P$ is good if and only if $F^\poOps$ is bounded.
  \qed
  \end{enumerate}
\end{lemma}

We leave the proof of this lemma to the reader.

The points $\delta_u$ defined above are vertices of $P^\polar$, more precisely, they are the vertices of the face $S^\poOps$ of $P^\polar$.  

\subsection{Definitions of the polyhedral complexes}

We consider the set of faces of $S^\polar$ which do not contain a vertex corresponding to a \textit{non-negativity inequality $x_e \ge 0$} for $e\in\Esetn$.  In
symbols, if $N$ denotes the set of these vertices of $S^\polar$, we deal with the polyhedral complex
\begin{equation}\label{eq:tsp-del-s}%
  \dl(N,S^\polar) := \{ F \text{ face of $S^\polar$} \mid F\cap N = \emptyset \} = \dl(\{\{x\}\mid x\in N\},\Cplx(S^\polar)).
\end{equation}

\subsubsection{Tight triangularity}

A \textit{triangle rooted at $u$} is a pair $u,vw$ consisting of a vertex $u\in\Vsetn$ and an edge $vw\in\Esetn$ not incident to $u$.  Let $a\in\RR^\Esetn$.  We
say that $a$ is \textit{\label{txt:def:metric}metric,} if it satisfies the triangle inequality, i.e., $t_{u,vw}(a) := a_{vu}+a_{uw}-a_{vw} \ge 0$ for all rooted
triangles $u,vw$.  Note that this implies $a_e \ge 0$ for all $e$.  We follow \cite{NadRina93} in calling $a$ \textit{tight triangular (TT),} if it is metric and
for each $u\in\Vsetn$ there exists $v,w$ such that the triangle inequality for this rooted triangle is satisfied with equation: $t_{u,vw}(a)=0$.  Abusively, we say
that a linear inequality is metric, or TT, if the left hand side vector has the property.

\subsubsection{Metric cone, TT-fan and flat TT-fan}

The \textit{metric cone,} $C=C_n$, consists of all \textit{(semi-)metrics} on $\Vsetn$.  In our context, a (semi-)metric is a point $d\in\RR^\Esetn$ which
satisfies the \textit{triangle inequality}
\begin{equation}\label{eq:triang-ieq}
  d_{vu}+d_{uw} - d_{vw} \ge 0
\end{equation}
for all distinct $u,v,w \in \Vsetn$.  Thus, in the terminology just defined, a semi-metric is just a metric point.
If we now let $F_{u,vw}$ denote the face of $C$ defined by inequality \eqref{eq:triang-ieq} we define the \textit{TT-fan} as follows:
\begin{equation}\label{eq:expo:TTcplx-preimg}
  \TTcplx' := \bigcap_{u\in\Vsetn} \bigcup_{v,w\ne u} \Cplx(F_{u,vw}) \quad\subset \Cplx(C).
\end{equation}
$\TTcplx'$ is a fan.  ``TT'' stands for ``tight triangular'', a term coined by Naddef \& Rinaldi \cite{NadRina93} for a point's property of being in
$\sabs{\TTcplx'}$.  However, we are not aware of any reference to this fan in the literature.  Heuristically, the elements of $\sabs{\TTcplx'}$ are metrics on
$\Vsetn$ satisfying the following: for every point $u\in\Vsetn$, there exist two other points $v,w\in\Vsetn$ such that $u$ is ``middle point'' of the ``line
segment'' between $v$ and $w$.

\label{def:p}Denote by $p\colon \RR^\Esetn \to \Lsp$ the orthogonal projection.  We will prove in the next section (Lemma~\ref{lem:TTfans}) that applying $p$
to $\TTcplx'$ produces a fan $\TTcplx$ isomorphic to $\TTcplx'$:
\begin{equation}\label{eq:def-TTcplx}
  \TTcplx := \{ p(F) \mid F\in \TTcplx' \}.
\end{equation}
We call $\TTcplx$ the \textit{flat TT-fan}.

\subsubsection{Definition of the edge sets $E^u(a)$}

Let $a\in S^\polar$.  For every $u\in\Vsetn$, we let $E^u(a)$ be the set of edges on which the slack of the triangle inequality~\eqref{eq:triang-ieq} is minimized:
\begin{multline}\label{eq:expo:Eu}
  E^u(a) := \Bigl\{ vw \in\Esetn \Bigm| u\ne v,w\text{, \ and }\\
  a_{vu}+a_{vw}-a_{vw} = \min_{v',w'\ne u} a_{v'u}+a_{uw'}-a_{v'w'} \Bigr\}.
\end{multline}
% The minimum in~\eqref{eq:expo:Eu} may be negative.

% The property that $\pi$, probably restricted to a subcomplex of the complex of $P^\polar$, becomes injective, is a strong fact.  It means that the complex
% of $P^\polar$ can be ``flattened'' to obtain a refinement of the complex of $S^\polar$.  In Section~\ref{sec:tsp}, we will undertake some effort to prove
% injectivity for the case of Traveling Salesman Polyhedra, a fact which is used to prove the completeness of a description by linear inequalities of the Graphical
% Traveling Salesman Polyhedron for nine cities \cite{OsRlTheis07,TheisParsimonious}.

\subsubsection{The TT-sub-complex of $P^\polar$}

Finally, we define a sub-complex of $\Cplx(P^\polar)$ consisting of all TT-points of $P^\polar$.  This sub-complex is what remains of the complex
$\bar\Cplx(P^\polar)$ of bounded faces of $P^\polar$ after deleting the conjugate face of $S$ in $P^\polar$, in symbols $\dl(S^\poOps,\bar\Cplx(P^\polar))$.

It will become clear in the next section (see Remark~\ref{rem:sbcplx-bd-tt}) that the points of the complex $\dl(S^\poOps,\bar\Cplx(P^\polar))$ are precisely the
points in $\sabs{\bar\Cplx(P^\polar)}$ which are tight triangular.

\subsection{Rotation and statements of the results}

We now give the rigorous definition of ``rotation'' and of the rotation complex, as outlined in the introduction.  More accurately, we define a ``rotation
partition'' of $\sabs{\dl(N,S^\polar)}$, which will turn out to be a polyhedral complex subdividing $\dl(N,S^\polar)$.

A point $a \in S^\polar$ corresponds to an inequality $a\iprod x \ge a\iprod z -1$ valid for $S$.  Rotating this inequality amounts to adding an equation valid for
$S$.  The left-hand side $q$ of such an equation is a linear combination of the left-hand sides of the equations $\delta_u\iprod x = 1$, and the right-hand side
coincides with $q\iprod z$.  Hence, for a fixed $q$, rotating the inequality $a\iprod x \ge a\iprod z -1$ by $q$ gives the following 
\begin{equation}\label{eq:rotated-ieq}
  (a+q)\iprod x \ge  a\iprod z -1 + q\iprod z.
\end{equation}

For $a \in \sabs{\dl(N,S^\polar)}$, let $\Pfaces(a) \in \Cplx(P)$ be the set of faces of $P$ which can be defined by the rotated version of the inequality
corresponding to $a$.  More precisely, a set $F\subset\RR^\Esetn$ is in $\Pfaces(a)$ if, and only if, there exists a $q$ as above, such that the rotated
inequality~\eqref{eq:rotated-ieq} is valid for $P$, and $F$ is the set of points in $P$ satisfying it with equality: $F = \{x\in P\mid (a+q)\iprod x = a\iprod z
-1 + q\iprod z \}$.

Now we define a partition $\TiltCplx^\circ$ of $\sabs{\dl(N,S^\polar)}$, by letting two points $a,b$ be in the same cell of $\TiltCplx^\circ$ if and only if
$\Pfaces(a) = \Pfaces(b)$.  Moreover, let $\TiltCplx$ be the set of all closures of cells of $\TiltCplx^\circ$:
\begin{equation*}
  \TiltCplx := \{ \close X \mid X \in \TiltCplx^\circ \}.
\end{equation*}
We call $\TiltCplx$ the \textit{rotation complex} (the word ``complex'' is justified by the following theorem).

\begin{theorem}\label{thm:expo:char}
  $\TiltCplx$ is a polyhedral complex.  Moreover, $X\mapsto\close X$ and $F\mapsto\relint F$ are inverse bijections between $\TiltCplx^\circ$ and $\TiltCplx$.
  The following is true.
  \begin{enumerate}\alphenumi
  \item\label{thm-enum:expo:char:TTfan} %
    The rotation complex $\TiltCplx$ is the common refinement of $\dl(N,S^\polar)$ and the flat TT-fan $\TTcplx$.
  \item\label{thm-enum:expo:char:E} %
    Two points $a$, $b$ in $\sabs{\dl(N,S^\polar)}$ are in the relative interior of the same face of the rotation complex $\TiltCplx$ if, and only if, they are in
    the relative interior of same face of $S^\polar$ and $E^u(a) = E^u(b)$ for all $u\in\Vsetn$.
  \end{enumerate}
\end{theorem}

This corresponds to item~(\ref{intro:mainresutls:charac}) on page~\pageref{intro:mainresutls:charac} in the introduction, while the next theorem corresponds to
item~(\ref{intro:mainresutls:inj}).

\begin{theorem}\label{thm:expo:inj}
  There is a projective homeomorphism $\pi\colon \sabs{\dl(S^\poOps,\bar\Cplx(P^\polar))} \to \sabs{\dl(N,S^\polar)}$ which induces a combinatorial equivalence
  between the polyhedral complex $\dl(S^\poOps,\bar\Cplx(P^\polar))$ and the rotation complex $\TiltCplx$.
\end{theorem}

\begin{remark}\label{rem:vertices-bij}
  Let us speak of a TT-vertex of $P^\polar$, if the point is TT, or, equivalently, if the vertex corresponds to a TT-facet of $P$.  Similarly, let us call a
  TT-vertex of $P^\polar$ an NR-vertex (non-NR-vertex), if the corresponding facet of $P$ is an NR-facet (non-NR-facet, resp.).
  Theorems \ref{thm:expo:char} and~\ref{thm:expo:inj} imply that the NR-vertices of $P^\polar$ are in bijection with the vertices of $\dl(N, \Stspn^\polar)$ via
  $\varphi$, while the non-NR vertices of $\Gtspn^\polar$ are mapped to non-vertex points by $\varphi$.
\end{remark}

\subsection{Parsimonious property of relaxations and the ridge graph}

Given a system $Bx \ge b$ of linear inequalities which are valid for $S$, one may ask how the minimum value of a linear function $x\mapsto c^\Tp x$ changes if
either degree inequalities or degree equations are present, in other words, whether the following inequality is strict:
\begin{subequations}\label{eq:parsi-ieq}
  \begin{eqnarray}
    \label{eq:parsi-ieq:ge}
    \min \bigl\{ c^\Tp x &\bigm|& Bx \ge b\text{, }\ \delta_v\iprod x \ge 1\, \forall v\text{, }\ x \ge 0 \bigr\}
    \\\notag\le\\
    \label{eq:parsi-ieq:eq}
    \min \bigl\{ c^\Tp x &\bigm|& Bx \ge b\text{, }\ \delta_v\iprod x  =  1\, \forall v\text{, }\ x \ge 0 \bigr\}
\end{eqnarray}
\end{subequations}
We say that the system of linear inequalities and equations in~\eqref{eq:parsi-ieq:ge},
\begin{equation}\label{eq:R_B}%\tag{$*B*$}
  \begin{aligned}[c]
      Bx &\ge b\\
      \delta_v\iprod x &\ge 1\; \forall v \in \Vsetn\\
      x &\ge 0
    \end{aligned}
\end{equation}
is a \textit{relaxation of $S$}.  Such a relaxation is said to have the \textit{parsimonious property} \cite{GoemBertim93} if equality holds in
\eqref{eq:parsi-ieq} for all $c$ satisfying the triangle inequality.

Goemans \cite{Goemans95} raised the question whether all relaxations of $S$ consisting of inequalities defining NR-facets of $P$ (in other words, they
are facet-defining for $P$ and for $S$) have the parsimonious property.  
% Since it was shown by Naddef \& Rinaldi \cite{NadRina93} that $S$ is the
% face of $P$ obtained by intersecting the facets defined by the degree inequalities, requiring that the inequalities define facets of $P$ only ensures
% that the inequalities have the ``right form''.

The parsimonious property had earlier been proved to be satisfied for the relaxation consisting of all inequalities defining facets of $P$ by Naddef \&
Rinaldi \cite{NadRina91}, in other words: optimizing an objective function satisfying the triangle inequality over $P$ yields the same value as optimizing
over $S$.
The parsimonious property has been verified by Goemans and Bertsimas \cite{GoemBertim93} for the relaxation consisting of all non-negativity inequalities $x_e \ge
0$, $e\in\Esetn$, and all so-called subtour elimination inequalities. For every $S\subsetneq \Vsetn$ with $\abs S \ge 2$, the corresponding \textit{subtour
  elimination inequality}
\begin{equation}\label{eq:subtourelimieq}
  \sum_\sstack{uv\in\Esetn\\\abs{\{u,v\}\cap S}=1} x_{uv} \;\;\ge\; 2,
\end{equation}
is valid and facet-defining for $S$ (whenever $n\ge 5$) \cite{GroePadb79a,GroePadb79b}.

To our knowledge, the first example of a relaxation of $S$ which does not have the parsimonious property is due to Letchford \cite{Let05PCv}.  While the
inequalities which he used did not define a facet of $S$ or of $P$, in \cite{OswReiThe05,OsRlTheis07}, a family of inequalities defining facets of
$P$ was given which does not have the parsimonious property.

\paragraph{As an application of Theorems \ref{thm:expo:char} and~\ref{thm:expo:inj},}
we give a necessary condition for a relaxation of $S$ consisting of inequalities defining NR-facets of $P$ to have the parsimonious property.
The condition is based on connectivity properties of the ridge graph of $P$.
Recall that the \textit{ridge graph} $\mathcal G$ of $P$ is the graph whose vertex set consists of all facets of $P$ where two facets are adjacent if their
intersection has dimension $\dim P-2$, i.e., it is a \textit{ridge.}
We will relate this relaxation to the induced subgraph $\mathcal G_B$ of the ridge graph of $P$ which is obtained if all vertices corresponding to the
facets defined by inequalities in $\mathcal R_B$ are deleted.

\begin{theorem}\label{thm:parsi-result}%
  Suppose $Bx\ge b$ consists of inequalities defining NR-facets of $P$.
  If the relaxation~\eqref{eq:R_B} of $S$ has the parsimonious property, then every connected component of $\mathcal G_B$ contains vertices corresponding to
  NR-facets of $P$.
\end{theorem}

Thus, we link the optimization view given by the parsimonious property question with combinatorial properties of the a polyhedral complex $\Cplx(P)$, or, more
precisely, of $\dl(S^\poOps,\bar\Cplx(P^\polar))$.  In the proof, Theorem~\ref{thm:expo:inj} is used to ``flatten'' the latter complex, which then allows to using
a separating-hyperplane argument for constructing a path in the ridge graph.

%%% Local Variables: 
%%% mode: latex
%%% TeX-master: "paper.tsp.tex"
%%% fill-column: 163
%%% End: 

% \input{general}
% \input{proof1}
%% tiltcplx/arxiv/tsp.tex

\newcommand{\mkttpt}{\vartheta}
\newcommand{\mkttieq}{\tilde\vartheta}
\newcommand{\sct}{{s}}

% \newcommand{\Gtsp}[1]{{P_{#1}}}
% \newcommand{\Gtspn}{\Gtsp{n}}
% \newcommand{\Bl}[1]{\PptpP{#1}}
% \newcommand{\Bln}{P^\polar}
% \newcommand{\Stsp}[1]{S_{#1}}
% \newcommand{\Stspn}{\Stsp{n}}

% \newcommand{\Lsmmd}{L^\orth}

%%%
\section{Proofs for Theorems~\ref{thm:expo:char} and~\ref{thm:expo:inj}}\label{sec:tsp}

In~\ref{ssec:tsp:prelim}, we will need to discuss some properties of Symmetric and Graphical Traveling Salesman polyhedra.  Most of them are generalizations of
facts in the seminal papers by Naddef \& Rinaldi \cite{NadRina91,NadRina93}.  The proof of Theorems \ref{thm:expo:char} and~\ref{thm:expo:inj} then takes up
Subsections \ref{ssec:descr-cplx} and~\ref{ssec:piinv}.

As said before, we assume in the whole section that $S=S_n$ and $P=P_n$ with $n\ge 5$, because we require the technical fact that non-negativity inequalities
$x_e \ge 0$, for an $e\in\Esetn$, define facets of $S$, which is true if and only if $n\ge 5$, see \cite{GroePadb79a,GroePadb79b}.

\subsection{Preliminaries on connected Eulerian multi-graph polyhedra}\label{ssec:tsp:prelim}

Naddef \& Rinaldi \cite{NadRina93} proved that every facet of $S$ is contained in precisely $n+1$ facets of $P$: the $n$ degree facts and one additional facet.
This fact and its generalizations are useful for our purposes.  For the sake of completeness, we will sketch its proof, and introduce some of the tools for the
proofs of our main theorems along the way.

First we set up some notations.  Let $D$ be the $\Vsetn\times\Esetn$-matrix whose rows are the $\delta_u^\Tp$, $u\in\Vsetn$.  Recall from Section~\ref{def:p}
that $p$ is the orthogonal projection from $\RR^\Esetn$ onto $\Lsp = \ker D$.  Note that the orthogonal complement $\Lsp^\bot = \ker p$ of $L$ is equal to $\img
D^\Tp = \{D^\Tp \xi \mid \xi \in \RR^\Vsetn\}$, the space of all linear combinations of the $\delta_u$.

In the following lemma, we summarize basic facts about tight triangularity.

\begin{lemma}\label{lem:use-TT-form}\mbox{}%
  \begin{aenumeratei}
  \item A metric inequality which is valid for $S$ is also valid for $P$.
  \item An inequality defining a good face of $P$ is metric.
  \item An inequality defining a good face $F$ of $P$ is TT if and only if $F$ is not contained in a degree facet.
  \item\label{lem:use-TT-form:nononneg-inherit} If a face $F$ of $P$ is good, then $S\cap F$ is also good.
  \item\label{lem:use-TT-form:codim-faces} Let the TT inequality $a\iprod x \ge 1$ be valid for $P$.  If it defines a face of co-dimension $c$ of $S$, then it
    defines a face of co-dimension at most $c$ of $P$.
  \item\label{lem:use-TT-form:TT-rep} For every $a\in\RR^\Esetn$ there is a unique TT representative in the co-set $a+\Lsp^\bot = \{ a + D^\Tp \xi \mid \xi
    \in \RR^\Vsetn\}$.  More precisely, we can obtain a unique $\lambda(a)\in \RR^\Vsetn$ for which $a-D^\Tp\lambda(a)$ is TT by letting
    \begin{equation}\label{eq:def-lambda-u}%
      \lambda_u(a) := \min_{v,w\ne u} t_{u,vw}(a)
    \end{equation}
  \end{aenumeratei}
\end{lemma}

\begin{wrapfigure}{r}{0mm}
  \parbox{21mm}{\scalebox{.5}{\input{shortcut-diff-c1-h.pstex_t}}\vspace*{-3ex}}
\end{wrapfigure}
Given a vertex $u$ and an edge $vw$ not incident to $u$, a \textit{shortcut} is a vector $\sct_{u,vw}:=\chi^{vw}-\chi^{vu}-\chi^{uw}\in\RR^\Esetn$.

\begin{proof}[Proofs for Lemma~\ref{lem:use-TT-form} (sketches).]
  The proofs of these facts are easy generalizations of arguments which can be found in \cite{NadRina93}.  

  The key ingredient in (a--c) is the \textit{shortcut argument} which Naddef \& Rinaldi pioneered in \cite{NadRina93}.  Let $x\in\ZZ_+^\Esetn$ represent the
  edge multi-set of a connected Eulerian multi-graph $H$ with vertex set $\Vsetn$.  If $H$ is not a cycle, i.e., if $H$ has a vertex $u$ of degree four or more,
  then one can find an edge $vw$ such that $vu$ and $vw$ are in $H$, and $H' := H\cup\{vw\}\setminus\{vu,vw\}$ is still a connected Eulerian multi-graph;
  cf.~the picture on the right.
  %% Fig.~\ref{fig:shortcut}
  If $y$ represents its edge multi-set, then $y = x + \sct_{u,vw}$.  This gives (a), the implication ``$\Rightarrow$'' in (c), and by carefully selecting the
  edge $vw$, (d).  Similarly, one can subtract a shortcut from an $x$, which gives (b), the other direction in (c), and, by taking for each vertex $u$ a
  shortcut $\sct_{u,vw}$, implies~(e).
  
  Item~(f) is straightforward computation.
\end{proof}

% \begin{figure}[tp]
%   \centering
%   \input{shortcut-diff-c1-h.pstex_t}
%   \caption{The shortcut argument}
%   \label{fig:shortcut}
% \end{figure}%

We now prove the important theorem of Naddef \& Rinaldi.

\begin{theorem}[\cite{NadRina93}]\label{thm:simplex}\mbox{}%
  \begin{aenumeratei}
  \item If a facet $G$ of $P$ contains $S$, then $G$ is a degree facet.
  \item Let $F$ be a good facet of $S$.  There exists a \textit{unique} facet $G$ of $P$ with $F = G\cap S$.
  \end{aenumeratei}
\end{theorem}
\begin{proof}
  \textit{(a). } If $G\supset S$, then $G$ is good by definition.  If $G$ is not equal to a degree facet, then, by Lemma~\ref{lem:use-TT-form}(c), it is
  defined by a TT inequality, which contradicts Lemma~\ref{lem:use-TT-form}(e).
  
  \textit{(b). } Clearly, $G$ exists because $S$ is a face of $P$.  Let $G$ be defined by an inequality $a\iprod x\ge\alpha$.  Then $a$ is TT by
  Lemma~\ref{lem:use-TT-form}(c), hence, by Lemma~\ref{lem:use-TT-form}(f), unique in the set $a + L^\orth$ of all left hand sides of inequalities
  defining the facet $F$ of $S$.
\end{proof}

\subsubsection{Related aspects of the polar polyhedra}

Theorem~\ref{thm:simplex}(b) can be rephrased as follows.  If $a$ is a vertex of $P^\polar$ such that the inequality $a\iprod x\ge1$ defines a facet of $S$,
then $a$ and $\delta_u$, $u\in\Vsetn$, are the vertices of an $n$-simplex which is a face of $P^\polar$.

\begin{remark} \label{rem:sbcplx-bd-tt}%
  By Lemma~\ref{lem:prop-blockingpoly}(b) and Lemma~\ref{lem:use-TT-form}(c), the points of the complex $\dl(S^\poOps,\bar\Cplx(P^\polar))$ are precisely
  the points in $\sabs{\bar\Cplx(P^\polar)}$ which are tight triangular.
\end{remark}

%%%%%%%%%%%%%%%%%%%%%%%%%%%%%%%%%%%%%%%%%%%%%%%%%%%%%%%%%%%%%%%%%%%%%%%%%%%%%%%%%%%%%%%%%%%%%%%%%%%%%%%%%%%%%%%%%%%%%%%%%%%%%%%%%%%%%%%%%%%%%%%%%%%%%%

\subsection{Descriptions of the rotation complex}\label{ssec:descr-cplx}

We will now prove Theorem~\ref{thm:expo:char}.
We start by 
% giving the precise definition for the flat TT-fan $\TTcplx$ drafted in Section~\ref{sec:expo} and then prove
proving that the two refinements of $\dl(N,S^\polar)$ defined in (a) and (b) respectively of Theorem~\ref{thm:expo:char} are identical: the one using the flat
TT-fan defined in~\eqref{eq:def-TTcplx} and the one using the sets $E^u(a)$ defined in~\eqref{eq:expo:Eu}.

Let us first verify that the orthogonal projection $p$ maps the TT-fan $\sabs{\TTcplx'}$ bijectively onto $L$.  For this, we define some mappings, based
on~\eqref{eq:def-lambda-u}:
\begin{alignat}{3}
  \notag
  \lambda_u\colon& \RR^\Esetn \to \RR \colon&& a \mapsto \min_{v,w\ne u} t_{u,vw}(a), \\
  \label{eq:def-lambda}
  \lambda\colon& \RR^\Esetn \to \RR^\Vsetn \colon&& a \mapsto (  \lambda_1(a),\dots,\lambda_n(a)  )^\Tp,\\
  \notag%\label{eq:def-mktt}
  \mkttpt\colon&\RR^\Esetn\to\RR^\Esetn\colon&& a \mapsto a - D^\Tp\lambda(a),\\
  \notag%\label{eq:def-mktt-ieq}
  \mkttieq\colon&\RR\times\RR^\Esetn\to\RR\times\RR^\Esetn\colon&& (\alpha,a) \mapsto (\alpha - \One\iprod\lambda(a), \mkttpt(a)).
\end{alignat}

\begin{lemma}\label{lem:TTfans}
  The mappings $p\colon \sabs{\TTcplx'} \to L$ and $\restr{\mkttpt}{L}\colon L \to \sabs{\TTcplx'}$ are inverses of each other.
\end{lemma}
\begin{proof}
  By Lemma~\ref{lem:use-TT-form}(\ref{lem:use-TT-form:TT-rep}), every co-set $a+\Lsp^\bot$ of $\Lsp^\bot$ contains a unique TT point, namely $\mkttpt(a)$.  The
  co-set also contains a unique point of $\Lsp$, namely the orthogonal projection $p(a)$ of $a$ onto $\Lsp$.  Hence, the two mappings are inverses of each
  other.
\end{proof}

In view of Lemma~\ref{lem:TTfans}, $p$ transports the fan $\TTcplx'$ into a fan $\TTcplx := p(\TTcplx')$ in $\Lsp$, the flat TT-fan defined in
Section~\ref{sec:expo}.  It is a complete fan in the ambient space $L$.
The next lemma states that the refinements of $\dl(N,S^\polar)$ used in Theorem~\ref{thm:expo:char} are identical.  The proof is a direct verification based on
the definitions of $E^u(\cdot)$ and $\mkttpt$, using Lemma~\ref{lem:TTfans}.

\begin{lemma}\label{lem:tiltcplx-refine}
  For two points $a,b\in L$, the following are equivalent:
  \begin{enumerate}\itromenumi
  \item\label{lem-enum:tiltcplx-refine:Eu} $E^u(a) = E^u(b)$ for all $u\in\Vsetn$
  \item\label{lem-enum:tiltcplx-refine:TTfan} $a$ and $b$ are in the relative interior of the same face of the flat TT-fan $\TTcplx$.
  \qed
  \end{enumerate}
\end{lemma}

For easy reference, let $\cplxD$ denote the common refinement of $\dl(N,S^\polar)$ and the flat TT-fan $\TTcplx$.  This is certainly a polyhedral complex, and
the previous lemma implies that two points are in the relative interior of the same face of $\cplxD$ if and only if (\ref{lem-enum:tiltcplx-refine:Eu}) holds.

This shows that items (\ref{thm-enum:expo:char:TTfan}) and~(\ref{thm-enum:expo:char:E}) of Theorem~\ref{thm:expo:char} are equivalent.
Moreover, to establish Theorem~\ref{thm:expo:char}, it remains to prove that the partition of $\sabs{\dl(N,S^\polar)}$ into open faces of $\cplxD$ coincides
with the partition $\TiltCplx^\circ$: Once this is established, both the statement about the closures and relative interiors in Theorem~\ref{thm:expo:char}, and
items (a) and (b) follow.

To prove that these two partitions coincide, we need to descent deeper into the properties of $P$.
If $F$ is a face of $P$, then a shortcut is said to be \textit{feasible} for $F$, if it is contained in the space $\direc F$.
We note the following for easy reference.

\begin{lemma}\label{lem:feasible-shortcut}
  If $F$ is a good face of $P$, then a shortcut $\sct_{u,vw}$ is feasible for $F$ if and only if $a\iprod \sct_{u,vw} = 0$ for one (and hence for all) $a\in
  \relint F^\poOps$.
\end{lemma}
\begin{proof}
  If $F$ is a good face, then the polarity relations of Lemma~\ref{lem:prop-blockingpoly} hold between $F$ and $F^\poOps$.  The details are left to the reader.
%   Let $a\in\relint F^\poOps$.  If $\sct_{u,vw}$ is a feasible shortcut for $F$, then, by Lemma~\ref{lem:prop-blockingpoly} and the definition of $\direc F$ in
%   Section~\ref{sec:prelim}, we have $a\iprod\sct_{u,vw}=0$.
%
%   If, on the other hand, $a\iprod \sct_{u,vw} = 0$ holds and $F$ is a good face, then $b\iprod\sct_{u,vw}=0$ for all $b\in F^\poOps$, because, by
%   Lemma~\ref{lem:use-TT-form}(b), the inequality $a\iprod \sct_{u,vw} \ge 0$ is valid for $F^\poOps$.  But the elements of $F^\poOps$ generate the space of all
%   (left-hand sides of) linear equations valid for $F$.
\end{proof}

The following lemma highlights the importance of shortcuts in the relationship between $S$ and $P$.

\begin{lemma}\label{lem:face-of-gtsp-uniq-determ-ham-shtcts}%
  A good face $F$ of $P$ is uniquely determined by
  \begin{itemize}
  \item the set of cycles whose characteristic vectors are contained in $F$, plus
  \item the set of its feasible shortcuts.
  \end{itemize}
\end{lemma}
\begin{proof}
  By the shortcut argument, every vertex of $F$ is either itself a cycle, or it can be constructed from a cycle by successively subtracting feasible shortcuts.
  Further, $\RR_+\,\chi^{uv}$ is a ray of $F$ if and only if, for any $a\in \relint F^\poOps$, we have $a_{uv}=0$ (by Lemma~\ref{lem:use-TT-form}b).  By
  Lemma~\ref{lem:feasible-shortcut}, this is equivalent to the property that for every $w\ne u,v$, both $\sct_{u,vw}$ and $\sct_{v,uw}$ are feasible
  shortcuts.
\end{proof}

We can now finish the proof of Theorem~\ref{thm:expo:char}.

\begin{proof}[Proof of Theorem~\ref{thm:expo:char}(\ref{thm-enum:expo:char:E})]
  Let $a\in \sabs{\dl(N,S^\polar)}$.  The inequalities of the form~\eqref{eq:rotated-ieq} all define good faces of $P$, because $a$ defines a face of $S$ not
  contained in a non-negativity facet of $S$.  Moreover, since every inequality of the form~\eqref{eq:rotated-ieq} defines the same face of $S$,
  Lemma~\ref{lem:face-of-gtsp-uniq-determ-ham-shtcts} implies that every member of the set $\Pfaces(a)$ of faces of $P$ defined by inequalities of the
  form~\eqref{eq:rotated-ieq} is uniquely determined by its set of feasible shortcuts.

  We claim that the set $\Pfaces(a)$ is in bijection with the set of all subsets of $\Vsetn$, where the bijection is accomplished in the following way: To a
  subset $I\subset\Vsetn$, there is a face in $\Pfaces(a)$ whose set of feasible shortcuts is precisely
  \begin{equation}\label{eq:cup-of-Es}\tag{$*$}
    \bigcup_{u\in I}  \{ \sct_{u,e} \mid e \in E^u(a) \}.
  \end{equation}
  The faces obtainable in this way are clearly pairwise distinct by what we have just said (note that $E^u(a)\ne\emptyset$).  We have to construct a
  corresponding inequality for every set $I$, and we have to show that all faces in $\Pfaces(a)$ can be reached in this way.

  For the former issue, for $I\subset\Vsetn$ we define $q := \sum_{u\not\in I} \delta_u$, and consider the inequality
  \begin{equation*}
    (\mkttpt(a)+q) \iprod x   \ge    -1 + a\iprod z - \One\iprod\lambda(a) + q \iprod z,
  \end{equation*}
  which is of the form~\eqref{eq:rotated-ieq} because $\One = Dz$, and defines a good face of $P$ whose set of feasible shortcuts is easily verified to be
  \eqref{eq:cup-of-Es}, by Lemma~\ref{lem:feasible-shortcut}.
  
  To see that every face in $\Pfaces(a)$ can be obtained in this way, it is easy to check, invoking Lemma~\ref{lem:use-TT-form} and the definition of
  $E^u(a)$, that, if there exists an edge $vw$ such that $\sct_{u,vw}$ is feasible for a face $F$ in $\Pfaces(a)$, then $vw\in E^u(a)$ and $\sct_{u,e}$ is
  feasible for $F$ for every $e\in E^u(a)$.

  This completes the proof of Theorem~\ref{thm:expo:char}.
\end{proof}

\subsection{Projective equivalence of the two complexes}\label{ssec:piinv}

We now proceed to prove Theorem~\ref{thm:expo:inj}.  We want to define a mapping $\pi$ by letting
\begin{subequations}\label{eq:def-pi}
\begin{equation}\label{eq:def-pi-of-a}
  \pi(a) := \frac{1}{a\iprod z - 1}\,p(a),
\end{equation}
for $a \in P^\polar$.  The denominator will be zero, if, and only if, $a\iprod x \ge 1$ is satisfied by equality for all $x\in S$, in other words, $\pi(a)$ is
well-defined for all $a\in P^\polar\setminus S^\poOps$.

By Lemma~\ref{lem:use-TT-form}, a point $a$ in the complex $\bar\Cplx(P^\polar)$ of bounded faces of $P^\polar$ defines a good face of $S$, so we have
$\pi(a)\in \sabs{\dl(N,S^\polar)}$, whenever $a \not \in S^\poOps$.  Hence, we have the mapping
\begin{equation}\label{eq:range-pi}%
  \pi \colon \sabs{   \dl(S^\poOps,\bar\Cplx(P^\polar))  } \to    \sabs{\dl(N,S^\polar)}
\end{equation}
\end{subequations}

In this subsection, we will prove that $\pi$ as given in~\eqref{eq:def-pi} is a homeomorphism, and show that it induces a combinatorial equivalence between
$\dl(S^\poOps,\bar\Cplx(P^\polar))$ and the rotation complex $\TiltCplx$; i.e., we prove Theorem~\ref{thm:expo:inj}.
%%
% Being the corresponding statement of Lemma~\ref{lem:gen:two-defs-eq}, the latter fact is no surprise: the only difference comes from the specific definition of
% the polar $P^\polar$ for the unbounded polyhedron $P$.
%%
We will explicitly construct the inverse mapping $\pi^{-1}$, which, essentially, transforms a point into its TT-representative in the sense of
Lemma~\ref{lem:use-TT-form}(\ref{lem:use-TT-form:TT-rep}).

When we write the projective mapping $\pi$ as a linear mapping from $\RR\times\RR^\Esetn \to \RR\times\Lsp$ as in Section~\ref{sec:prelim}, it has the following
form:
\begin{equation*}
  \tilde\pi :=
  \begin{pmatrix}
    -1 &  z\iprod\place \\
    0 & p
  \end{pmatrix}.
\end{equation*}

As a technical intermediate step in the construction of $\pi^{-1}$, we define a linear mapping $I\colon \RR\times\RR^m\to\RR\times\RR^m$ taking points in
$\RR\times\Lsp$ to points in $\RR\times\RR^\Esetn$ by the matrix
\begin{equation*}
  I:=
  \begin{pmatrix}
    -1 & z\iprod\place \\
    0 & \id
  \end{pmatrix},
\end{equation*}
Now we let $(\gamma,c) := \mkttieq\circ I (1,\place)$; in long:
\begin{multline} \label{eq:c-gamma-def}
  (\gamma,c)\colon a \mapsto (\gamma(a),c(a)) := \mkttieq(I(1,a))
%   \\
  = \bigl(-1 + a\iprod z -  \One\iprod\lambda(a)\,,\;  a -D^\Tp\lambda(a) \bigr).
\end{multline}

Clearly, for all $a\in\Lsp$, the point $c(a)$ is TT.  If $a\in S^\polar$, i.e., if the inequality $a\iprod x \ge -1 + a\iprod z$ is valid for $S$, then the
inequality $c(a) \ge \gamma(a)$ is of the form \eqref{eq:rotated-ieq} (cf.~the corresponding statement in the proof of Theorem~\ref{thm:expo:char} above).  We
note the following fact as a lemma for the sake of easy reference.

\begin{lemma}\label{lem:cgamma-def-face}
  If $a\in S^\polar$, the two inequalities $a\iprod x\ge -1+a\iprod z$ and $c(a)\iprod x\ge\gamma(a)$ define the same face of $S$.
  \qed
\end{lemma}

Finally, we define
\begin{equation}  \label{eq:phi-def}
  \varphi\colon
  \sabs{\del(N,S^\polar)} \to  \sabs{\del(S^\poOps,\bar\Cplx(P^\polar))}
  \colon
  \quad
  a \mapsto \frac{1}{\gamma(a)} \; c(a).
\end{equation}

\begin{proof}[Proof of Theorem~\ref{thm:expo:inj}]
  In the remainder of this section, we will discuss the following issues:
  \begin{aenumeratei}
  \item $\varphi$ is well-defined (in \ref{sssec:phi-welldef})
  \item $\varphi$ is a left-inverse of $\pi\colon \sabs{ \del(S^\poOps,\bar\Cplx(P^\polar)) } \to \sabs{ \del(N,S^\polar) }$ (in
    \ref{sssec:phi-leftinverse-of-pi})
  \item $\pi\colon \sabs{ \del(S^\poOps,\bar\Cplx(P^\polar))} \to \sabs{ \del(N,S^\polar) }$ is onto (in \ref{sssec:phi-inj})
  \item $\pi\colon \sabs{ \del(S^\poOps,\bar\Cplx(P^\polar)) } \to \sabs{ \del(N,S^\polar) }$ is a refinement map inducing the rotation complex $\TiltCplx$ (in
    \ref{sssec:refinement-map}).
  \end{aenumeratei}
  Items (b) and (c) imply that
  \begin{equation*}
    \varphi\circ \pi = \id_{\sabs{\del(S^\poOps,\bar\Cplx(P^\polar))}}
    \text{\quad and\quad}
    \pi\circ \varphi = \id_{\sabs{\del(N,S^\polar)}}.
  \end{equation*}
  From this and (d), Theorem~\ref{thm:expo:inj} follows.
\end{proof}

%%%%%%%%%%%%%%%%%%%%%%%%%%%%%%%%%%%%%%%%%%%%%%%%%%%%%%%%%%%%%%%%%%%%%%%%%%%%%%%%%%%%%%%%%%%%%%%%%%%%

\subsubsection{$\pi$ induces the rotation complex}\label{sssec:refinement-map}

We first prove that $\pi$ is a refinement map inducing the rotation complex.  For this, we use the above stated properties inverse mapping $\varphi$, which are
only proved below.

\begin{lemma}\label{lem:phi-F-G}
  For every face $F$ of $\TiltCplx$ there exists a face $G$ of $\dl(S^\poOps,\bar\Cplx(P^\polar))$ with $\varphi(\relint F) \subset \relint G$.
\end{lemma}
\begin{proof}
  Let $F'$ be the face of $\dl(N,S^\polar)$ with $\relint F \subset \relint F'$.  Now, let $a\in\relint F$ and $G^\#$ be the face of $P$ defined by the
  inequality $\varphi(a)\iprod x \ge 1$.  Since this inequality defines the same face of $P$ as the inequality $c(a)\iprod x \ge \gamma(a)$ which is of the
  form~\eqref{eq:rotated-ieq}, the set of cycles whose characteristic vectors are in $G^\#$ coincides with those contained in the face $F'^\poOps$ of $S$, where
  the conjugate face is taken in $S$ vs.~$S^\polar$ (not in $P$ vs.~$P^\polar$), and thus does not depend on the choice of $a \in \relint F'$.  Moreover, the
  set of feasible shortcuts for $G^\#$ is in bijection with $E^u(a)$, $u\in\Vsetn$, and hence, by Theorem~\ref{thm:expo:char}, depends only on $F$ not on the
  choice of $a\in\relint F$.  Thus, by Lemma~\ref{lem:face-of-gtsp-uniq-determ-ham-shtcts}, $G^\#$ does not depend on the choice of $a\in\relint F$.  Hence, with
  $G := (G^\#)^\poOps$, we have $\varphi(a) \in \relint G$ for all $a\in\relint F$.
\end{proof}

Lemma~\ref{lem:phi-F-G} provides us with a mapping $\Phi\colon F\mapsto G$ with $F$ and $G$ as in the lemma.  Moreover, the argument based on
Lemma~\ref{lem:face-of-gtsp-uniq-determ-ham-shtcts} in the proof of Lemma~\ref{lem:phi-F-G} shows that $\Phi(F_1) \ne \Phi(F_2)$ whenever $F_1\ne F_2$, i.e.,
$\Phi$ is injective, and, by the surjectivity of $\varphi$, it is also onto.  Hence, we obtain the following:

\begin{lemma}\label{lem:phi-Phi-F-G}
  There is a bijection $\Phi\colon \TiltCplx \to \dl(S^\poOps,\bar\Cplx(P^\polar))$ with $\Phi(F) = \varphi(F)$.
\end{lemma}
\begin{proof}
  What remains to be shown is that $\Phi(F) = \varphi(F)$.  We already know that $\varphi(\relint F) \subset \relint \Phi(F)$.  Standard Euclidean topology
  arguments show that $\varphi$ maps the boundary $\partial F$ of $F$ into the boundary of $\varphi(F)$.  (This is most easily seen by noting that $\varphi$ is
  the inverse of a projective mapping.)  But the boundary of $F$ is the union of its facets, so we have
  \begin{equation*}
    \Phi(F)\setminus\relint\varphi(F)
    \supset
    \varphi(F)\setminus\relint\varphi(F)
    \supset
    \varphi(\partial F)
    =
    \bigcup\nolimits_{F'} \varphi(F')
    \subset
    \bigcup\nolimits_{F'} \Phi(F')
  \end{equation*}
  where the union extends over all facets $F'$ of $F$.  Consequently, by the injectivity of $\Phi$, we have $\varphi(\partial F) \subset \partial\Phi(F)$.
  Again by standard topological arguments (Borsuk-Ulam theorem) and the injectivity of $\Phi$ this implies $\varphi(\partial F) = \partial\Phi(F)$, and this in
  turn gives $\varphi(F) = \Phi(F)$.
\end{proof}

\begin{remark}
  The topological arguments contained in the proof of Lemma~\ref{lem:phi-Phi-F-G} can be replaced by more technical polyhedral theory ones.  In any case, they
  reflect basic geometric facts which are not worth to be emphasized.
\end{remark}

%%%%%%%%%%%%%%%%%%%%%%%%%%%%%%%%%%%%%%%%%%%%%%%%%%%%%%%%%%%%%%%%%%%%%%%%%%%%%%%%%%%%%%%%%%%%%%%%%%%%
\subsubsection{We show: $\varphi$ is well-defined} \label{sssec:phi-welldef}

We start by showing that the quotient in \eqref{eq:phi-def} is well-defined.
The key ingredient here is the fact that we are only considering good faces.

\begin{lemma}  \label{lemma:gamma-never-vanish}
  For all $a\in \sabs{\del(N,S^\polar))}$ we have $\gamma(a)>0$.
\end{lemma}
\begin{proof}
  Assume to the contrary that $\gamma(a)=0$.  Since $c(a)$ is metric, $c(a)\ge 0$ holds.  We distinguish two cases: $c(a)=0$ and $c(a)\gneq 0$.  In the first
  case, the hyperplane defined by $c(a)\iprod x = \gamma(a)$ contains $S$, while $a\iprod x\ge -1+a\iprod z$ defines a proper face of $S$, a contradiction to
  Lemma~\ref{lem:cgamma-def-face}.  On the other hand, if $c(a)\gneq 0$, then the inequality $c(a)\iprod x\ge\gamma(a)$ is a non-negative linear combination of
  non-negativity inequalities, and hence the face defined by $c(a)\iprod x = \gamma(a)$ is contained in a non-negativity facet of $P$.  But since
  $a\in\sabs{\del(N,S^\polar))}$, i.e., $a$ it is not a relative interior point of a face of $S^\polar$ which contains a vertex of $S^\polar$
  corresponding to a non-negativity facet of $S$, the face of $S$ defined by $a\iprod x \ge -1+ a\iprod z$ is not contained in a non-negativity facet of $S$.
  Thus Lemma~\ref{lem:cgamma-def-face} yields a contradiction.
\end{proof}

It remains to be shown that the image of $\sabs{\del(N,S^\polar))}$ under $\varphi$ is really contained in the target space given in \eqref{eq:phi-def}:
For all $a\in \sabs{\del(N,S^\polar))}$ we have $\varphi(a) \in \sabs{\del(S^\poOps,\bar\Cplx(P^\polar))}$.
This also follows from Lemma~\ref{lem:cgamma-def-face}: The inequality $\varphi(a)\iprod x\ge 1$ is valid for $P$, and the face it defines is good.  Since
$\varphi(a)$ is TT, the conclusion follows from Remark~\ref{rem:sbcplx-bd-tt}.

%%%%%%%%%%%%%%%%%%%%%%%%%%%%%%%%%%%%%%%%%%%%%%%%%%%%%%%%%%%%%%%%%%%%%%%%%%%%%%%%%%%%%%%%%%%%%%%%%%%%
\subsubsection{We show: $\varphi$ is a left-inverse of $\pi$,}\label{sssec:phi-leftinverse-of-pi}%
i.e., for all $a\in\sabs{\del(S^\poOps,\bar\Cplx(P^\polar))}$ the identity $\varphi(\pi((a))=a$ holds.

\begin{lemma}
  For all $a\in \sabs{\del(S^\poOps,\bar\Cplx(P^\polar))}$ we have $(\gamma,c)(\tilde\pi(1,a)) = (1,a)$.  In particular, we have that $\varphi\circ \pi$
  restricted to $\abs{\del(S^\poOps,\bar\Cplx(P^\polar))}$ is equal to the identity mapping on this set.
\end{lemma}
\begin{proof}
  To see this we compute
  \begin{multline*}
    I(\tilde\pi(1,a))
    = I(-1+a\iprod z,p(a))
    \\
    = \bigl(1-a\iprod z - z\iprod p(a), p(a) \bigr)
    \\
    = \bigl((p(a)-a)\iprod z +1, p(a) \bigr)
  \end{multline*}
  Using that $a$ is TT (Remark~\ref{rem:sbcplx-bd-tt}), we conclude
  \begin{equation*}
    \mkttieq(I(\tilde\pi(1,a))) = \Bigl( (p(a)-a)\iprod z +1 - \lambda(p(a))\iprod \One,\;  a\Bigr).
  \end{equation*}
  Since $a$ is TT, by Lemma~\ref{lem:use-TT-form}(\ref{lem:use-TT-form:TT-rep}), $\lambda(p(a))$ is a solution to $p(a)-a = D^\Tp\lambda$.  Thus, using
  $\One=Dz$, it follows that
  \begin{equation*}
    (p(a)-a)\iprod z + 1 - \lambda(p(a))\iprod \One = (p(a)-a)\iprod z + 1 - D^\Tp \lambda(p(a))\iprod z = 1.
  \end{equation*}

  From the statement about $(\mkttieq\circ I)\circ \tilde\pi$, the statement about the projective mappings $\varphi\circ\pi$ follows by a slight generalization
  of the well-known fact that concatenation of projective mappings corresponds to multiplication of the respective matrices
  (Remark~\ref{rem:prelim:prj-map-concat}).  We omit the computation, and only note that it makes use of the fact that the two mappings
  $h_1\colon a\mapsto a - D^\Tp\lambda(a)$ and $h_2\colon a\mapsto a\iprod z + \lambda(a)\iprod\One$
  are positive homogeneous, i.e., $h_i(\eta a)=\eta h_i(a)$ for $\eta\ge0$, $i=1,2$, which follows directly from the definition of $\lambda$.
\end{proof}

%%%%%%%%%%%%%%%%%%%%%%%%%%%%%%%%%%%%%%%%%%%%%%%%%%%%%%%%%%%%%%%%%%%%%%%%%%%%%%%%%%%%%%%%%%%%%%%%%%%%
\subsubsection{We show: $\varphi$ is one-to-one} \label{sssec:phi-inj}

Since we already know that $\varphi\circ \pi=\id$, surjectivity of $\pi$ is equivalent to injectivity of $\varphi$.
% Because of the particular definition of the polar of $P$, this does not follow from Lemma~\ref{lem:pi-onto} in Section~\ref{sec:general}.
It is actually easier to prove the following slightly stronger statement.

\begin{lemma} \label{lemma:c-gamma-inj}%
  Let $a,b\in \Lsp$.  If there exists an $\eta\in\RR_+$ such that $(\gamma(a),c(a))=\eta(\gamma(b),c(b))$ then $\eta=1$ and $a=b$.
  In particular, $\varphi$ is injective.
\end{lemma}

\begin{proof}
  Let such $a,b,\eta$ be given.  We have
  \begin{multline*}
    0 %
    = c(a) - \eta c(b) %
    = a - D^\Tp\lambda(a) - \eta\Bigl[ b - D^\Tp\lambda(b) \Bigr]
    \\
    = a - \eta b - D^\Tp\Bigl[ \lambda(a) - \eta \lambda(b) \Bigr].
  \end{multline*}
  Since $a,b\in\Lsp$ and $D^\Tp[\lambda(a) - \eta \lambda(b)] \in \Lsp^\bot$ we have
  \begin{equation}  \label{eq:lemma:c-gamma-inj:1} \tag{$\ast$}%
    a-\eta b = 0 =D^\Tp\lambda(a) - \eta D^\Tp\lambda(b)
  \end{equation}
  Applying $z\iprod\place$ to the second equation, we obtain
  \begin{equation*}
    0 = \One\iprod\lambda(a) - \eta\; \One\iprod\lambda(b)
  \end{equation*}
  Applying this to the $\gamma$s, we have
  \begin{multline*}
    0
    = \gamma(a)-\eta\gamma(b)
    \\
    =
    -1 + a\iprod z - \One\iprod \lambda(a)
    - \eta\Bigl[-1 + b\iprod z - \One\iprod\lambda(b) \Bigr]
    \\
    =
    -1 + \eta + (a - \eta b)\iprod z.
  \end{multline*}
  Since $z \in \Lsp^\bot$ we have $(a - \eta b)\iprod z=0$, whence $\eta=1$.  Now $a=b$ follows from
  \eqref{eq:lemma:c-gamma-inj:1}.
\end{proof}

% We extract the following fact from the proof up to \eqref{eq:lemma:c-gamma-inj:1}.
% \begin{corollary}\label{cor:c-inj}
%   $c\colon \Lsp \to \RR^\Esetn$ is one-to-one.
% \end{corollary}

%%%%%%%%%%%%%%%%%%%%%%%%%%%%%%%%%%%%%%%%%%%%%%%%%%%%%%%%%%%%%%%%%%%%%%%%%%%%%%%%%%%%%%%%%%%%%%%%%%%%%%%%%%%%%%%%%%%%%%%%%%%%%%%%%%%%%%%%%%%%%%%%%%%%%%%%%%%%%%%%
%%%%%%%%%%%%%%%%%%%%%%%%%%%%%%%%%%%%%%%%%%%%%%%%%%%%%%%%%%%%%%%%%%%%%%%%%%%%%%%%%%%%%%%%%%%%%%%%%%%%%%%%%%%%%%%%%%%%%%%%%%%%%%%%%%%%%%%%%%%%%%%%%%%%%%%%%%%%%%%%
%%% Local Variables: 
%%% mode: latex
%%% TeX-master: "paper.tsp.tex"
%%% fill-column: 160
%%% End: 

% subdiv/parsimonious.tex
% TeX-master: paper.tsp.tex

\section{Proof of Theorem~\ref{thm:parsi-result}}\label{sec:parsi}

We will apply Theorem~\ref{thm:expo:inj} to prove Theorem~\ref{thm:parsi-result}.  The following lemma is the link between parsimonious property and geometry.

\begin{lemma}\label{lemma:GTSP:nonNR-not-writable-by-Bieqs-only}
  Let $Bx\ge \One$ be a system of inequalities defining NR-facets of $P$ such that the relaxation $\mathcal R_B$ has the parsimonious property.
  If $c^\Tp x\ge \gamma$ defines a non-NR facet of $P$, then $c,\gamma$ cannot be written in the form 
  \begin{equation}\label{eq:mu-form}
    \begin{aligned}
      c        &= b - \sum\nolimits_{v\in\Vsetn} \mu_v d_v \\
      \gamma   &= \beta - \sum\nolimits_{v\in\Vsetn} \mu_v
    \end{aligned}
  \end{equation}
  with $b^\Tp = \sum_j t_j b_j$ a non-negative linear combination of rows $b_j$ of $B$, $\beta = \sum_j t_j$, and $\mu_v \in \RR$ for all $v\in\Vsetn$.
\end{lemma}
\begin{proof}
  Suppose that $c,\gamma$ can be written as in \eqref{eq:mu-form}.  Then minimizing the cost function $c$ over the relaxation consisting of
  \begin{itemize}
  \item all non-negativity inequalities
  \item all degree equations(!) $\delta_v\iprod x = 1$, $v\in \Vsetn$;
  \item all inequalities in the system $Bx\ge \One$.
  \end{itemize}
  yields $\gamma$ as the minimum.  If the degree equations are relaxed to inequalities, then, by the parsimonious property of $\mathcal R_B$, the minimum is
  still $\gamma$.  By Farkas's Lemma, this implies that the inequality $c\iprod x\ge\gamma$ is dominated by non-negativity inequalities, degree inequalities, and
  inequalities in $Bx\ge \One$.  This is impossible since $(c,\gamma)$ defines a non-NR facet of $P$ and all facets in $Bx\ge \One$ are NR.
\end{proof}

We are now ready to prove the Theorem~\ref{thm:parsi-result}.

\begin{proof}[Proof of Theorem~\ref{thm:parsi-result}]
  Let $a_\circ\iprod x \ge 1$ be an inequality defining a non-NR facet of $P$ which is not in the system $Bx\ge \One$.  By Lemma~\ref{lem:prop-blockingpoly},
  the paths in the ridge graph of $P$ not touching non-negativity facets are precisely the paths in the 1-skeleton of $P^\polar$.

  Thus, we have to find a path in the graph of $P^\polar$ which starts from $a_\circ$, ends in an NR-vertex, and does not use any degree vertices or vertices
  corresponding to rows of $B$.

  By Theorem~\ref{thm:expo:inj}, we know that there exists a projective homeomorphism $\pi\colon \sabs{\dl(S^\poOps,\bar\Cplx(P^\polar))} \to
  \sabs{\dl(N,S^\polar)}$ transporting the polyhedral complex $\dl(S^\poOps,\bar\Cplx(P^\polar))$ onto the rotation complex.  As in the proof of that theorem,
  we let $\varphi := \pi^{-1}$.

  Let $a:=\varphi^{-1}(a_\circ)$.  This point is contained in the relative interior of a unique face $F$ of $S^\polar$ containing no non-negativity vertex.  Let
  $\cplxD_F$ denote the set of all faces of the rotation complex $\cplxD$ which are contained in $F$, and let $B_F$ denote the set of vertices $b$ of $F$ for
  which $\varphi(b)^\Tp$ is a row of $B$.  We will prove the following:

  \begin{claim}\label{claim:1}
    Let $F$ be a face of $\dl(N, S^\polar)$, and let $a$ be a relative interior point of $F$ which is a vertex of $\mathcal D_F$ such that $\varphi(a)^\Tp$
    is not a row of $B$.  Then there is a path in the 1-skeleton of $\mathcal D_F$ starting at $a$, ending in a vertex of $F$, and not touching any of the
    vertices in $B_F$.
  \end{claim}

  By Theorem~\ref{thm:expo:char}, this claim implies the existence of the desired path in the graph of $P^\polar$ and thus concludes the proof of
  Theorem~\ref{thm:parsi-result}.
\end{proof}

\begin{proof}[Proof of Claim~\ref{claim:1}]
  The proof of the claim is by induction on $\dim F$.  For $\dim F=0$, we are done, because then $a$ is a vertex of $F$.  Let $\dim F\ge 1$, and assume the
  claim holds for relative interior points $a'$ of faces $F'$ with dimension $\dim F' < \dim F$.

  If $B=\emptyset$, we are done.  Otherwise let $Q := \conv B_F$.  This is a non-empty polytope which is contained in $F^\poOps$.  Using
  Lemma~\ref{lemma:GTSP:nonNR-not-writable-by-Bieqs-only} we will show the following:

  \begin{claim}\label{claim:2}
    Let $c$ be a vertex of $\mathcal D_F$ which is not a member of $B_F$.  Then $c$ cannot be contained in $Q$.
  \end{claim}
    
  The proof of Claim~\ref{claim:2} is technical, and we postpone it till the proof of Claim~\ref{claim:1} is finished.  If Claim~\ref{claim:2} is true, however,
  then we we know that $a$ is not in $Q$.
  Let $p,\pi$ define a hyperplane separating $a$ from $Q$, i.e., $q\iprod p < \pi$ for all $q\in Q$, and $a\iprod p > \pi$.  See Fig.~\ref{fig:stepQ} for an
  illustration.  It assumes the face $F$ is an 8-gon.

\begin{figure}[ht]
  \centering
  \scalebox{.5}{\input{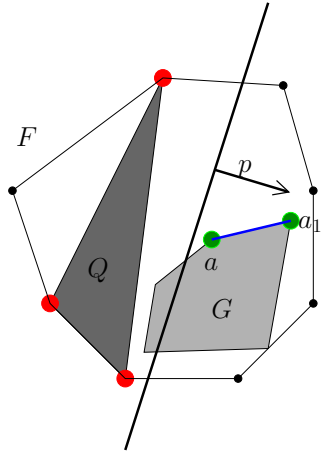}}
  \caption{One step of the path}
  \label{fig:stepQ}
\end{figure}

  By a standard general position argument, we can assume that $p$ is not parallel to any face with co-dimension at least one in $\mathcal D_F$.
  Hence, there exists an $\eps>0$ such that the line segment $a + ]0,\eps[\cdot p$ is contained in the relative interior of a $\dim F$-dimensional face $G$ of
  $\mathcal D_F$, of which $a$ is a vertex.
  By elementary polytope theory (the edges of a polyhedron incident to a fixed vertex span a cone of the same dimension as the polyhedron), $G$ must have a
  vertex $a_1$ adjacent to $a$ with $a\iprod p < a_1\iprod p$.  Clearly $a_1\not\in B_F$.
  
  If $a_1$ is in the boundary of $F$, then the induction hypotheses implies the existence of a path from $a_1$ to a vertex of $F$ not using any vertex in $B_F$.
  If that is not the case, we apply the argument in the previous paragraph inductively to obtain a path $a,a_1,\dots,a_k$ in the 1-skeleton of $\mathcal D_F$
  with $a\iprod p < a_1\iprod p < \dots < a_{j}\iprod p < a_{j+1}\iprod p < \dots < a_k\iprod p$.  Since the 1-skeleton of $\mathcal D_F$ is finite and the path
  we are constructing is $p$-increasing, a vertex on the boundary of $F$ will eventually be reached.

  This concludes the proof of Claim~\ref{claim:1}.
\end{proof} % of Claim 1

\begin{proof}[Proof of Claim~\ref{claim:2}]
  Let $c$ be a vertex of $\mathcal D_F$ with $c\not\in B_F$.  Assume that $c\in\conv B_F$, i.e., $c$ can be written as a convex combination $c=\sum_{j=1}^k t_j
  b_j$ with $\varphi(b_j)^\Tp$ a row of $B$ for all $j=1,\dots,k$.  Clearly, $c$ cannot be a vertex of $F$, so $\varphi^{-1}(c)\iprod x \ge 1$ defines a non-NR
  facet of $P$ by Remark~\ref{rem:vertices-bij}.  We compute
  \begin{multline*}
    c - \sum_{v\in\Vsetn} \lambda_v(c) d_v
    =\\
    \sum_j t_j \lt( b_j - \sum_{v\in\Vsetn} \lambda_v(b_j) d_v \rt)
    - \sum_{v\in\Vsetn} \lt( \lambda_v(c) - \sum_j t_j \lambda_v(b_j) \rt) d_v.
  \end{multline*}
  Letting $\sigma := 1-\sum_v \lambda_v(c)$, $\tau_j := 1-\sum_v \lambda_v(b_j)$, and $\mu_v := \lambda_v(c) - \sum_j t_j \lambda_v(b_j)$, we see that
  \begin{align*}
    \sigma \varphi(c) &= \sum_j t_j \tau_j \varphi(b_j) - \sum_v \mu_v d_v\\
    \sigma            &= \sum_j t_j \tau_j - \sum\nolimits_{v\in\Vsetn} \mu_v
  \end{align*}
  This means that the inequality $\sigma \varphi(c)\iprod x \ge \sigma$ can be written as a non-negative linear combination of the inequalities
  $\varphi(b_j)\iprod x \ge 1$, $j=1,\dots,k$ plus a linear combination of degree vertices as in \eqref{eq:mu-form}.  Since the former inequality defines a
  facet of $P$ by Theorems \ref{thm:expo:char} and~\ref{thm:expo:inj}, and the inequalities forming the non-negative linear combination are taken from the
  system $Bx \ge \One$, Lemma~\ref{lemma:GTSP:nonNR-not-writable-by-Bieqs-only} yields a contradiction.
\end{proof}

%%%%%%%%%%%%%%%%%%%%%%%%%%%%%%%%%%%%%%%%%%%%%%%%%%%%%%%%%%%%%%%%%%%%%%%%%%%%%%%%%%%%%%%%%%%%%%%%%%%%%%%%%%%%%%%%%%%%%%%%%%%%%%%%%%%%%%%%%%%%%%%%%%%%%%%%%%%%%%%%
%%%%%%%%%%%%%%%%%%%%%%%%%%%%%%%%%%%%%%%%%%%%%%%%%%%%%%%%%%%%%%%%%%%%%%%%%%%%%%%%%%%%%%%%%%%%%%%%%%%%%%%%%%%%%%%%%%%%%%%%%%%%%%%%%%%%%%%%%%%%%%%%%%%%%%%%%%%%%%%%
%%% Local Variables: 
%%% mode: latex
%%% TeX-master: "paper.tsp.tex"
%%% fill-column: 160
%%% End: 

\section{Outlook}

% I am intrigued to know the homotopy type of the complex $\del(N,\Cplx(S^\polar))$, with $S$ a Symmetric Traveling Salesman Polytope.

% It would be interesting whether there are more relations between the combinatorics of the boundary complex of $P^\polar$, particularly as higher dimensional
% faces are concerned.  As for edges, 
We conjecture that the necessary condition for parsimonious property in Theorem~\ref{thm:parsi-result} is also sufficient.

\begin{conjecture*}
  If every connected component of $\mathcal G_B$ contains vertices corresponding to NR-facets of $\Gtspn$, then the relaxation $\mathcal R_B$ of has the
  parsimonious property.
\end{conjecture*}

The conjecture holds for the known relaxations of $S$ consisting of NR-inequalities described in \cite{OsRlTheis07} which fail the parsimonious property.

\section*{Acknowledgments}

The author would like to thank the \textit{Deutsche Forschungsgemeinschaft}, DFG, for funding this research, and the \textit{Communaut\'e fran\c caise de
  Belgique -- Actions de Recherche Concert\'ees} for supporting the author during the time the paper was written down.

Moreover, thanks are extended to Jean-Paul Doignon and Samuel Fiorini, U.L.B., for helpful discussions on the topic of this paper, and to Marcus Oswald for
inspiring discussions on the topic of Section~\ref{sec:parsi}.

\end{document}